\documentclass[11pt]{amsart}

\usepackage[margin=1.1in]{geometry}
\usepackage{amsmath,amssymb,amsthm,mathtools}
\usepackage{hyperref}
\usepackage{orcidlink}

\usepackage{fourier}
\newtheorem{theorem}{Theorem}
\newtheorem{lemma}[theorem]{Lemma}
\newtheorem{proposition}[theorem]{Proposition}
\newtheorem{corollary}[theorem]{Corollary}
\theoremstyle{remark}
\newtheorem{remark}[theorem]{Remark}

\DeclareMathOperator{\Tr}{Tr}
\DeclareMathOperator{\tr}{tr}
\DeclareMathOperator{\ESD}{ESD}
\DeclareMathOperator{\Var}{Var}
\newcommand{\C}{\mathbb{C}}
\newcommand{\R}{\mathbb{R}}
\newcommand{\E}{\mathbb{E}}

\newcommand{\Z}{\mathbb{Z}}

\newcommand{\eur}{\mathrm{e}}

\title[]{Double-scaling limits for monitored Haar products}
\author[]{Joon Hyung Lee\,\orcidlink{0000-0001-8628-9922}}
\thanks{LIACS, Leiden University. Supported by the Dutch Research Council (NWO) under the project \emph{Boosting the Search for New Quantum Algorithms with AI} (BoostQA), file number NGF.1623.23.033, research programme Quantum Technologie 2023. J.H.L. also acknowledges the support of the Swiss National Science Foundation grant no.~200020-182517.}
\address{Leiden Institute of Advanced Computer Science (LIACS), Leiden University,
         Einsteinweg 55, 2333 CC Leiden, The Netherlands}
\email{j.h.lee@liacs.leidenuniv.nl}
\date{\today}


\begin{document}

\begin{abstract}
We study products of independent Haar matrices interspersed with a fixed
corank-one projection. As the matrix dimension $N$ and product length $L$
grow with $L/N\to\tau$, we prove that the empirical distribution of squared
singular values converges to the free multiplicative Poisson law with
$S$-transform $S(z)=\exp(\tau/(1+z))$, for both the unitary and orthogonal
ensembles. Lagrange inversion yields its moments explicitly, proving
Beenakker's conjectured Erlang formula for every moment order.
\end{abstract}
\maketitle

%-----------------------------------------------------------------------
\section{Introduction}
%-----------------------------------------------------------------------

Consider a quantum system with $N$ orthogonal modes, called channels. At each
step, unitary evolution mixes these modes, after which a measurement
distinguishes a fixed monitored subspace from its orthogonal complement. We
retain the component in the unmonitored subspace, represented by an orthogonal
projection $P$. After $L$ steps, the surviving amplitude is described by
\[
  B_L=(PS_L)(PS_{L-1})\cdots(PS_1),
\]
where $S_1,\ldots,S_L$ are independent Haar unitary matrices in $U(N)$.

For a unit input state $v$, the probability of passing through all monitoring
steps is
\[
  \|B_Lv\|^2=v^\dagger B_L^\dagger B_Lv.
\]
An orthonormal eigenbasis of $B_L^\dagger B_L$ therefore provides distinguished
input states, called transmission eigenchannels. Their eigenvalues are the
corresponding transmission probabilities.

In the corank-one case, $\dim\ker P=1$, so exactly one channel is monitored at
each step. We investigate the accumulated effect when the number of steps is
comparable to the number of channels, $L\sim\tau N$.
Beenakker \cite{Beenakker2501} computed the limiting normalized trace moments,
denoted by $m_p(\tau)$,
for integer orders $1\le p\le6$ in the form
\[
  m_p(\tau)=\eur^{-p\tau}\frac{G_p(\tau)}{(p-1)!}.
\]
He conjectured the closed form
\begin{equation}\label{eq:beenakker-conjecture}
  G_p(\tau)
  =
  (1-\tau)(p-1)!\sum_{i=0}^{p-1}\frac{(p\tau)^i}{i!}
  +\tau^pp^{p-1}.
\end{equation}

The starting point of this paper is the observation that this moment structure
is naturally encoded by the $S$-transform (defined in Section~\ref{sec:setup})
\[
  S_{\mu_t}(z)=\exp\!\left(\frac{t}{1+z}\right),
\]
which identifies the candidate limiting law as the free multiplicative
Poisson semigroup $(\mu_t)_{t\ge0}$. For $L_N/N\to\tau$, the candidate
limiting law is therefore $\mu_\tau$. This measure is known in free probability. See
Bercovici--Voiculescu \cite{BercoviciVoiculescu1992} and the recent work
\cite{JKM2025I,JKM2025II,Kabluchko2025}. Here we identify it with the
monitored-transport problem. Writing
$m_p(\tau)=\int_0^1 x^p\,d\mu_\tau(x)$ for its $p$-th moment, $p\ge1$,
Lagrange inversion gives
\begin{equation}\label{eq:intro-moment}
  m_p(\tau)
  =
  \eur^{-p\tau}
  \sum_{k=0}^{p-1}\left(1-\frac{k}{p}\right)
  \frac{(p\tau)^k}{k!},
\end{equation}
and a telescoping identity turns this into \eqref{eq:beenakker-conjecture} for
every $p$. We use the modern references of Gessel \cite{Gessel2016} and
Surya--Warnke \cite{SuryaWarnke2023} for Lagrange inversion.

We prove the simultaneous random-matrix limit $L\asymp N$ directly.
Fixed-$L$ asymptotic freeness does not, on its own,
control the regime in which the number of Haar factors grows with $N$. In the
corank-one model we show, for every fixed $p$ and $L_N/N\to\tau$,
\[
  \frac1N\Tr(B_{L_N}^\dagger B_{L_N})^p
  \xrightarrow{\mathrm{prob}}m_p(\tau),
\]
with bias $O(N^{-1})$ relative to $m_p(L_N/N)$ and variance $O(N^{-2})$.
The same bias bound holds relative to $m_p(\tau)$ if
$L_N/N=\tau+O(N^{-1})$. Hence we obtain weak convergence of the
empirical spectral distribution to $\mu_\tau$. The proof uses an exact
rank-one compression expansion, concentration on the sphere, and a triangular
limiting ODE for the rescaled time $t=n/N$, whose solution at time $t$ is the
moment sequence of $\mu_t$. Evaluating at $t=\tau$ gives the limiting moments
$m_p(\tau)$. The same argument works for the orthogonal ensemble.
Both symmetry classes are considered in \cite{Beenakker2501}.
A related growing-product regime
for $\beta$-Jacobi processes is studied by Ahn \cite{Ahn2022}, with emphasis
on fluctuations of the largest eigenvalues near the upper spectral edge. Our
result controls global normalized trace moments in the corank-one monitored
model.

The paper is organized as follows. Section~\ref{sec:setup} states the main
results. Section~\ref{sec:fixedL-proof} records the short fixed-length and
small-loss free-probability arguments needed for the identification.
Section~\ref{sec:erlang} derives the moment formula and proves the
conjectured closed form. Section~\ref{sec:cumulants} records known spectral facts and a
brief application to transmission statistics. Section~\ref{sec:double-scaling}
contains the main finite-$N$ argument and its orthogonal counterpart.

%-----------------------------------------------------------------------
\section{Setup and main results}\label{sec:setup}
%-----------------------------------------------------------------------

For each $N\ge 1$, let $P=P_N$ be a deterministic orthogonal projection on $\C^N$ of
rank $M_N$. Here $P_N^2=P_N=P_N^\dagger$, where $\dagger$ denotes the
conjugate transpose. We write $I$ for the identity matrix, $\Tr$ for the
ordinary matrix trace, and set $B_0=I$. We consider two asymptotic regimes
for $M_N$.

\emph{The fixed-length regime.} The number $L$ of factors and a loss fraction
$c\in(0,1)$ are fixed, while
\begin{equation}\label{eq:fixedL-regime}
  \frac{M_N}{N}\longrightarrow 1-c
  \qquad\text{as } N\to\infty.
\end{equation}
Thus a positive fraction of the channels is monitored at each step.

\emph{The double-scaling regime.} The projection $P$ has corank one, so $M_N=N-1$, while
\begin{equation}\label{eq:double-scaling-regime}
  \frac{L_N}{N}\longrightarrow\tau\ge0.
\end{equation}
The loss fraction per step is $1/N$. For $\tau>0$, the number of steps is
of order $N$.

By Haar invariance, the results
depend only on the rank $M_N$, so the particular choice of $P_N$ is immaterial.

Let
\[
  \lambda_1^{(N,L)},\ldots,\lambda_N^{(N,L)}
\]
denote the eigenvalues of the positive semidefinite matrix $B_L^\dagger B_L$,
which we call the transmission eigenvalues. Equivalently, they are the squared
singular values of $B_L$, denoted by $\sigma_k(B_L)$, with both lists
ordered nonincreasingly. They satisfy
\[
  \lambda_k^{(N,L)}=\sigma_k(B_L)^2.
\]
They lie in $[0,1]$, since each factor $PS_i$ is a contraction.

Writing $\delta_a$ for the unit point mass at $a$, the empirical transmission
eigenvalue distribution is
\[
  \mu_{N,L}
  =
  \frac1N\sum_{k=1}^N\delta_{\lambda_k^{(N,L)}}
  =
  \ESD(B_L^\dagger B_L).
\]

Let
\[
  \nu_c=(1-c)\delta_1+c\delta_0.
\]

For compactly supported probability measures on $[0,\infty)$, we write
$\rho\boxtimes\rho'$ for their free multiplicative convolution. This is the law of $A^{1/2}BA^{1/2}$,
where $A$ and $B$ are freely independent positive operators with laws $\rho$
and $\rho'$, respectively. Freeness means that the tracial state vanishes
on every alternating product of centered polynomials in $A$ and $B$.
The notation $\rho^{\boxtimes L}$ denotes the $L$-fold convolution of $\rho$
with itself. See \cite{MingoSpeicher2017} for background.

We fix the remaining notation. For a Borel probability measure $\rho$ on a compact subset of
$[0,\infty)$, let $m_p(\rho)=\int x^p\,d\rho(x)$ denote its $p$-th moment,
$p\ge1$. If $m_1(\rho)\ne0$, write
$\psi_\rho(z)=\sum_{p\ge1}m_p(\rho)z^p$ for its moment series. Since
$\psi_\rho'(0)=m_1(\rho)\ne0$, $\psi_\rho$ has a local compositional inverse $\psi_\rho^{-1}$ at the origin.
The $S$-transform is
\begin{equation}\label{eq:S-def}
  S_\rho(z)=\frac{1+z}{z}\,\psi_\rho^{-1}(z),
\end{equation}
where the value at $z=0$ is defined by analytic continuation. This is an
analytic function near $0$ and satisfies
$S_{\rho\boxtimes\rho'}=S_\rho S_{\rho'}$ when both first moments are nonzero. We also use the
$\Sigma$-transform $\Sigma_\rho(z)=S_\rho\bigl(z/(1-z)\bigr)$. For an
$N\times N$ Hermitian matrix $A$ we write $\ESD(A)=\frac1N\sum_k\delta_{
\lambda_k(A)}$ for its empirical spectral distribution, where $\lambda_k(A)$
are its eigenvalues, counted with multiplicity. We use $\|A\|$ for the
operator norm, $\|A\|_1=\Tr\sqrt{A^\dagger A}$ for the trace norm, and
$\|v\|$ for the Euclidean norm of a vector. For Hermitian matrices $A$ and
$B$, we write $A\le B$ if $B-A$ is positive semidefinite, equivalently if
$v^\dagger(B-A)v\ge0$ for every vector $v$. In particular, $0\le A\le I$
means that all eigenvalues of $A$ lie in $[0,1]$. We write $\E$ and $\Var$ for
expectation and variance. A subscript on $O$, as in $O_p(N^{-1})$, indicates
allowed dependence of the deterministic bound on that parameter.  Finally,
$\rho_N\xrightarrow{\mathrm{prob}}\rho$ for random probability measures on a
fixed compact interval means $d(\rho_N,\rho)\to0$ in probability, where $d$ is
any metric for the weak topology. Equivalently, $\int f\,d\rho_N\to\int
f\,d\rho$ in probability for every continuous $f$. We use
$\xrightarrow{w}$ for weak convergence of deterministic probability measures.

\begin{theorem}[Fixed-length monitored chains]\label{thm:fixedL}
Fix $L\ge 1$ and $c\in(0,1)$, and assume the fixed-length regime
\eqref{eq:fixedL-regime}. Then, as $N\to\infty$,
\[
  \ESD(B_L^\dagger B_L)
  \xrightarrow{\mathrm{prob}}
  \nu_c^{\boxtimes L}.
\]
The limit is identified by its $S$-transform, which by multiplicativity is
\[
  S_{\nu_c^{\boxtimes L}}(z)
  =
  \left(\frac{1+z}{1-c+z}\right)^L.
\]
\end{theorem}

\begin{theorem}[Free law of small monitored losses]\label{thm:free-small-loss}
Fix $\tau\ge 0$. For $n>\tau$, let $c_n=\tau/n$. Then
\[
  \nu_{c_n}^{\boxtimes n}
  \xrightarrow{w}
  \mu_\tau,
\]
where $\mu_\tau$ is the compactly supported probability measure on $[0,1]$
characterized by
\[
  S_{\mu_\tau}(z)
  =
  \exp\left(\frac{\tau}{1+z}\right).
\]
For $\tau=0$ this gives $\mu_0=\delta_1$, since $c_n=0$ and $\nu_0=\delta_1$.
The measures satisfy the semigroup identity \eqref{eq:semigroup-short}.
\end{theorem}

The measure $\mu_\tau$ is the free multiplicative Poisson law. See
\cite[\S7]{BercoviciVoiculescu1992} for its place in the classification of
$\boxtimes$-infinitely divisible measures on $[0,\infty)$, and
\cite{JKM2025I,JKM2025II,Kabluchko2025} for detailed studies. The convergence
in Theorem~\ref{thm:free-small-loss} is a special case of the general limit
theory for free multiplicative convolutions of infinitesimal arrays
\cite{BercoviciWang,ArizmendiHasebe2018}. Section~\ref{sec:free-small-loss}
gives a direct proof using the $S$-transform.

The following identity connects the free-probability moment formula with the
closed form conjectured in \cite[Eq.~(13a)]{Beenakker2501}.

\begin{theorem}[Erlang closed form]\label{thm:beenakker-closed-form}
Let $\mu_\tau$ be as in Theorem~\ref{thm:free-small-loss}, and write its
moments as
\[
  m_p(\tau)=\int_0^1 x^p\,d\mu_\tau(x)=\eur^{-p\tau}\frac{G_p(\tau)}{(p-1)!} .
\]
Then for every integer $p\ge 1$,
\[
  G_p(\tau)
  =
  (p-1)!
  \sum_{k=0}^{p-1}
  \left(1-\frac{k}{p}\right)
  \frac{(p\tau)^k}{k!}
  =
  (1-\tau)(p-1)!\sum_{i=0}^{p-1}\frac{(p\tau)^i}{i!}
  +
  \tau^p p^{p-1},
\]
as an identity of polynomials in $\tau$. In particular
\eqref{eq:beenakker-conjecture}, verified in \cite{Beenakker2501} for
$p\le 6$, holds for all $p$.
\end{theorem}

\begin{theorem}[Double-scaling limit]\label{thm:double-scaling}
Assume $P$ has rank $N-1$ and $L_N$ satisfies \eqref{eq:double-scaling-regime}.
Let $\mu_\tau$ be the measure of Theorem~\ref{thm:free-small-loss}.
Then, for every integer $p\ge1$,
\[
  \frac1N\Tr\bigl(B_{L_N}^\dagger B_{L_N}\bigr)^p
  \xrightarrow{L^2}m_p(\tau).
\]
Consequently,
\[
  \ESD(B_{L_N}^\dagger B_{L_N})\xrightarrow{\mathrm{prob}}\mu_\tau.
\]
Moreover, for any fixed positive integers $\alpha_1,\ldots,\alpha_r$,
\[
  \E\prod_{j=1}^r\frac1N\Tr\bigl(B_{L_N}^\dagger B_{L_N}\bigr)^{\alpha_j}
  \longrightarrow
  \prod_{j=1}^r m_{\alpha_j}(\tau).
\]
\end{theorem}

Theorem~\ref{thm:double-scaling} says that the random matrix chain with $O(N)$
independent monitored Haar factors is asymptotically described by the same
free multiplicative Poisson law obtained by first passing to free probability
and then taking the small-loss limit. Thus the simultaneous random-matrix
limit agrees with the law obtained from the fixed-length limit followed by
the free small-loss limit. The last assertion is the asymptotic factorization of normalized trace
products, which is a consequence of concentration.

%-----------------------------------------------------------------------
\section{Fixed-length and free small-loss limits}\label{sec:fixedL-proof}
%-----------------------------------------------------------------------

For a single step,
\[
  (PS)^\dagger(PS)=S^\dagger P S,
\]
so if $M_N/N\to1-c$ then its empirical law converges deterministically to
\[
  \nu_c=(1-c)\delta_1+c\delta_0.
\]
Moreover
\[
  \psi_{\nu_c}(z)=\frac{(1-c)z}{1-z},
  \qquad
  S_{\nu_c}(z)=\frac{1+z}{1-c+z}.
\]

\begin{proof}[Proof of Theorem~\ref{thm:fixedL}]
Set $A_n=B_nB_n^\dagger$. Then $A_0=I$, $0\le A_n\le I$, and
\[
  A_n=PS_nA_{n-1}S_n^\dagger P.
\]
For $n=1$, $A_1=P$, whose empirical law converges to $\nu_c$ by
\eqref{eq:fixedL-regime}. Suppose inductively that
$\ESD(A_{n-1})\to\nu_c^{\boxtimes(n-1)}$ in probability.
Conditionally on $A_{n-1}$, the matrix $S_n$ is independent Haar, so
asymptotic freeness with Haar conjugation gives
\[
  \ESD\bigl(PS_nA_{n-1}S_n^\dagger P\bigr)
  \xrightarrow{\mathrm{prob}}
  \nu_c\boxtimes\nu_c^{\boxtimes(n-1)}
  =\nu_c^{\boxtimes n}.
\]
Here we use the standard asymptotic-freeness theorem of
\cite[\S4.3]{MingoSpeicher2017}. See also the Haar integration estimates in
\cite{CollinsSniady2006}. The moment estimates are uniform for positive
contractions, so conditioning and then averaging also allows the random
matrix $A_{n-1}$ with empirical law converging in probability.
All measures are supported on $[0,1]$, so moment convergence implies weak
convergence. Since $L$ is fixed, induction requires only finitely many steps.
Finally, $B_LB_L^\dagger$ and $B_L^\dagger B_L$ have the same eigenvalues,
which proves the asserted convergence. Multiplicativity of the $S$-transform
gives
\[
  S_{\nu_c^{\boxtimes L}}(z)
  =\left(\frac{1+z}{1-c+z}\right)^L.
\]
\end{proof}

\subsection{The free small-loss limit}\label{sec:free-small-loss}
Set $c_n=\tau/n$.  Then
\[
  S_{\nu_{c_n}^{\boxtimes n}}(z)
  =\left(\frac{1+z}{1-\tau/n+z}\right)^n
  \longrightarrow
  \exp\!\left(\frac{\tau}{1+z}\right)
\]
locally uniformly on every compact set avoiding $z=-1$ (and, for each large
$n$, the moving pole $-1+\tau/n$).  Indeed,
\[
 n\log\left(\frac{1+z}{1-\tau/n+z}\right)
 =\frac{\tau}{1+z}+O(n^{-1})
\]
uniformly on such compact sets.

\begin{proof}[Proof of Theorem~\ref{thm:free-small-loss}]
Let $\psi_n=\psi_{\nu_{c_n}^{\boxtimes n}}$ be the moment series of
$\nu_{c_n}^{\boxtimes n}$. The inverse moment series satisfy
\[
  \psi_n^{-1}(w)
  =\frac{w}{1+w}S_{\nu_{c_n}^{\boxtimes n}}(w)
  \longrightarrow
  \frac{w}{1+w}\exp\!\left(\frac{\tau}{1+w}\right)
\]
locally uniformly near $0$.  Since the derivative at $0$ of the limit is
$e^\tau\ne0$, analytic inversion gives coefficientwise convergence of the
moment series.  All measures are supported on $[0,1]$, so the Hausdorff moment
problem is determinate and convergence of the moments implies weak
convergence to a unique probability measure $\mu_\tau$ with
\[
  S_{\mu_\tau}(z)=\exp\!\left(\frac{\tau}{1+z}\right).
\]
\end{proof}

The same formula immediately gives the semigroup identity
\begin{equation}\label{eq:semigroup-short}
  \mu_s\boxtimes\mu_t=\mu_{s+t},\qquad \mu_0=\delta_1.
\end{equation}
This is the free multiplicative Poisson semigroup of
\cite{BercoviciVoiculescu1992,JKM2025II,Kabluchko2025}.

%-----------------------------------------------------------------------
\section{Laguerre polynomials and Lagrange inversion}\label{sec:erlang}
%-----------------------------------------------------------------------

We now derive the moment formula for $\mu_\tau$ from the $S$-transform
\[
  S_{\mu_\tau}(z)=\exp\left(\frac{\tau}{1+z}\right).
\]
The combination of Lagrange inversion with a Laguerre generating function used
below parallels the computation of Demni \cite{Demni2017} for the free Jacobi
process associated with a single projection of normalized rank $1/2$.  There the Laguerre
polynomials enter through Biane's moment formula
$\phi(u_t^k)=\eur^{-kt/2}L^{(1)}_{k-1}(kt)/k$ for free unitary Brownian motion
\cite{Biane1997}, where $(u_t)$ is free unitary Brownian motion,
$\phi$ is its tracial state, and $L_n^{(\alpha)}$ denotes the generalized
Laguerre polynomial of degree $n$ and parameter $\alpha$. Here they produce the Erlang-type moments of $\mu_\tau$.
Our convention is
\[
  \sum_{n\ge0}L_n^{(\alpha)}(x)t^n
  =(1-t)^{-\alpha-1}\exp\!\left(-\frac{xt}{1-t}\right),
  \qquad |t|<1,
\]
with $L_n(x)=L_n^{(0)}(x)$.

Let
\[
  \psi_\tau(z)=\sum_{p\ge 1}m_p(\tau)z^p
\]
be the moment generating series of $\mu_\tau$. Since
\[
  S_{\mu_\tau}(w)=\frac{1+w}{w}\psi_\tau^{-1}(w),
\]
we have
\[
  \psi_\tau^{-1}(w)
  =
  \frac{w}{1+w}
  \exp\left(\frac{\tau}{1+w}\right).
\]
Equivalently,
\[
  \psi_\tau^{-1}(w)
  =
  \eur^\tau w
  \frac{\exp\left(-\frac{\tau w}{1+w}\right)}{1+w}.
\]

All identities in this section are identities of formal power series in $w$ or
$u$. The Laguerre series below also converges analytically for $|w|<1$.
We write $[z^p]F(z)$ for the coefficient of $z^p$ in a power series $F$.
Taking $\alpha=0$, we have
\[
  \sum_{n=0}^{\infty}L_n(x)t^n
  =
  \frac{1}{1-t}
  \exp\left(-\frac{xt}{1-t}\right),
\]
and setting $x=-\tau$, $t=-w$, we obtain
\[
  \frac{\exp\left(-\frac{\tau w}{1+w}\right)}{1+w}
  =
  \sum_{n=0}^{\infty}(-1)^nL_n(-\tau)w^n.
\]
Thus
\[
  \psi_\tau^{-1}(w)
  =
  \eur^\tau w
  \sum_{n=0}^{\infty}(-1)^nL_n(-\tau)w^n.
\]

This Laguerre expansion describes the inverse moment series. To obtain the
moments themselves, we apply Lagrange inversion
\cite{Gessel2016,SuryaWarnke2023}.

Write
\[
  z=\psi_\tau^{-1}(w)=w\varphi(w),
  \qquad
  \varphi(w)=\frac{1}{1+w}\exp\left(\frac{\tau}{1+w}\right).
\]
Then $w=\psi_\tau(z)$, and Lagrange inversion gives
\[
  m_p(\tau)
  =
  [z^p]\psi_\tau(z)
  =
  \frac1p[u^{p-1}]\varphi(u)^{-p}.
\]
Since
\[
  \varphi(u)^{-p}
  =
  (1+u)^p
  \exp\left(-\frac{p\tau}{1+u}\right),
\]
we get
\[
  m_p(\tau)
  =
  \frac1p[u^{p-1}]
  (1+u)^p
  \exp\left(-\frac{p\tau}{1+u}\right).
\]
Writing
\[
  \exp\left(-\frac{p\tau}{1+u}\right)
  =
  e^{-p\tau}
  \exp\left(\frac{p\tau u}{1+u}\right),
\]
we obtain
\[
  m_p(\tau)
  =
  \frac{\eur^{-p\tau}}p
  [u^{p-1}]
  (1+u)^p
  \exp\left(\frac{p\tau u}{1+u}\right).
\]
Expanding the exponential,
\[
  \exp\left(\frac{p\tau u}{1+u}\right)
  =
  \sum_{k=0}^{\infty}
  \frac{(p\tau)^k}{k!}
  \frac{u^k}{(1+u)^k}.
\]
Therefore
\[
  (1+u)^p
  \exp\left(\frac{p\tau u}{1+u}\right)
  =
  \sum_{k=0}^{\infty}
  \frac{(p\tau)^k}{k!}
  u^k(1+u)^{p-k}.
\]
Only $0\le k\le p-1$ contributes to the coefficient of $u^{p-1}$, and
\[
  [u^{p-1}]u^k(1+u)^{p-k}
  =
  [u^{p-1-k}](1+u)^{p-k}
  =
  p-k.
\]
Hence
\begin{equation}\label{eq:moment-formula}
  m_p(\tau)
  =
  \eur^{-p\tau}
  \sum_{k=0}^{p-1}
  \left(1-\frac{k}{p}\right)
  \frac{(p\tau)^k}{k!}.
\end{equation}

Equivalently,
\[
  m_p(\tau)=\eur^{-p\tau}\frac{G_p(\tau)}{(p-1)!},
\]
where
\[
  G_p(\tau)
  =
  (p-1)!
  \sum_{k=0}^{p-1}
  \left(1-\frac{k}{p}\right)
  \frac{(p\tau)^k}{k!}.
\]
This is the Erlang-type polynomial appearing in the moment recursion of \cite{Beenakker2501}.
Thus the Erlang structure comes from Lagrange inversion applied to the
Laguerre expansion of the inverse moment series.

\begin{remark}[Relation to a Laguerre formula]\label{rem:laguerre}
Kabluchko \cite[Prop.~3.5(a)]{Kabluchko2025} gives a Laguerre-polynomial
representation for the moments of the more general family
$\nu_{\beta,\gamma}$ with
$S$-transform $\exp(\gamma/(\beta+1+z))$.  Although the stated hypothesis
there excludes the value $\beta=0$ relevant here, the same coefficient
identity follows directly from Lagrange inversion at $\beta=0$, giving
\[
  m_p(\tau)
  =
  \frac{\eur^{-p\tau}}{p}
  \sum_{j=0}^{p-1}(-1)^j\binom{p}{j+1}L_j^{(-1)}(p\tau).
\]
This is equivalent to \eqref{eq:moment-formula}.
\end{remark}

\subsection{The Erlang closed form}

The Erlang closed form \eqref{eq:beenakker-conjecture} follows from a
telescoping identity.

\begin{proof}[Proof of Theorem~\ref{thm:beenakker-closed-form}]
The first equality is \eqref{eq:moment-formula}. For the second, put $x=p\tau$.
Then
\begin{equation}\label{eq:telescope}
\begin{aligned}
  \sum_{k=0}^{p-1}\left(1-\frac{k}{p}\right)\frac{x^k}{k!}
  &=\sum_{k=0}^{p-1}\frac{x^k}{k!}
    -\frac{x}{p}\sum_{k=0}^{p-2}\frac{x^k}{k!}\\
  &=\left(1-\frac{x}{p}\right)\sum_{k=0}^{p-1}\frac{x^k}{k!}
    +\frac{x^p}{p!},
\end{aligned}
\end{equation}
where the sum ending at $p-2$ is empty when $p=1$.
Multiplying by $(p-1)!$ gives the claimed identity.
\end{proof}

In \cite{Beenakker2501}, the closed form is stated in two equivalent ways, as
\eqref{eq:beenakker-conjecture} and as an incomplete gamma expression
\cite[Eq.~(13b)]{Beenakker2501}, with the correspondence to the computed
polynomials assumed for all $p$. For $p=1$, $G_1=1$ directly. The following
incomplete-gamma identities are used for $p\ge2$.
Theorem~\ref{thm:beenakker-closed-form}
covers both forms. Write $\Gamma(s,x)=\int_x^\infty t^{s-1}\eur^{-t}\,dt$ for
the upper incomplete gamma function. The identity
$(p-1)!\sum_{i=0}^{p-1}x^i/i!=\eur^{x}\Gamma(p,x)$ turns
\eqref{eq:beenakker-conjecture} into
$(1-\tau)\eur^{p\tau}\Gamma(p,p\tau)+\tau^pp^{p-1}$, and the recurrence
$\Gamma(p,x)=(p-1)\Gamma(p-1,x)+x^{p-1}\eur^{-x}$ then gives
\[
  G_p(\tau)
  =
  \eur^{p\tau}(1-\tau)(p-1)\Gamma(p-1,p\tau)+(p\tau)^{p-1},
\]
the two terms $\pm\tau^pp^{p-1}$ cancelling.

The coefficient triangle of the $G_p$ is sequence A122525 in the OEIS, where it
arises in the Erlang delay (Erlang C) formula for the $M/M/s$ queue
\cite{AbateWhitt2010}. Here the first $M$ denotes Poisson arrivals
(equivalently, independent exponential interarrival times), the second $M$
denotes independent exponential service times, and $s$ is the number of
parallel servers. The letter $M$ stands for Markovian, reflecting the
memoryless property of the exponential distribution. In the stationary
regime, the Erlang C formula gives the probability that an arriving customer
must wait for service. Thus the same polynomial family also appears naturally
from Lagrange inversion of the free multiplicative Poisson $S$-transform.

%-----------------------------------------------------------------------
\section{Known spectral facts and a transport application}\label{sec:cumulants}
\label{sec:spectral}
%-----------------------------------------------------------------------

The spectral description of $\mu_\tau$ is known. For $0<\tau<1$, $\mu_\tau$ has an atom
of mass $1-\tau$ at $1$, while for $\tau\ge1$ this atom is absent. Its
absolutely continuous part is supported on $(0,\tau e^{1-\tau})$ and admits
an explicit Lambert-$W$ density. See \cite[Thm.~4.2]{JKM2025I}.
The atom and the location of the spectral edge were also obtained directly
in the transport model in \cite[Sec.~V, Eq.~(17)]{Beenakker2501}.

As a brief application, the moments also determine the statistics of the
number of transmitted particles. Define the limiting cumulant generating
function per channel by
\[
  K_\tau(s)=\int_0^1\log\bigl(1+x(e^s-1)\bigr)\,d\mu_\tau(x),
  \qquad s\in\R,
\]
and let $\kappa_j(\tau)=\left.\frac{d^j}{ds^j}K_\tau(s)\right|_{s=0}$ be the corresponding particle-count
cumulants. These are ordinary cumulants of the counting statistics, rather
than free cumulants of $\mu_\tau$. The first two are
\[
  \kappa_1(\tau)=m_1(\tau)=e^{-\tau},
  \qquad
  \kappa_2(\tau)=m_1(\tau)-m_2(\tau)
  =(e^\tau-1-\tau)e^{-2\tau}.
\]
Thus the Fano factor is
\begin{equation}\label{eq:fano}
  \mathsf F(\tau)=\frac{\kappa_2(\tau)}{\kappa_1(\tau)}
  =1-(1+\tau)e^{-\tau}.
\end{equation}
%-----------------------------------------------------------------------
\section{The double-scaling regime}\label{sec:double-scaling}
%-----------------------------------------------------------------------

We now prove Theorem~\ref{thm:double-scaling} in the double-scaling regime
\eqref{eq:double-scaling-regime}, with $c_N=1/N$ and hence $L_Nc_N\to\tau$.
This is the scaling relevant to the monitored-product model of
\cite{Beenakker2501}.

\subsection{Reduction to a projection chain}

We first identify the spectral process with that of a projection chain.

\begin{lemma}\label{lem:reduction}
Let $P=I-ee^\dagger$ with $e$ a unit vector, let $S_1,S_2,\ldots$ be
independent Haar unitaries, and let $C_n=B_nB_n^\dagger$ with
$B_n=(PS_n)\cdots(PS_1)$. Define independently
\[
  \widetilde C_0=I,
  \qquad
  \widetilde C_{n}=\Pi_{n}\widetilde C_{n-1}\Pi_{n},
  \qquad
  \Pi_n=I-u_nu_n^\dagger,
\]
where $u_1,u_2,\ldots$ are independent and Haar-uniform on the unit sphere of
$\C^N$. Writing $\mathrm{spec}(A)$ for the eigenvalue list of a Hermitian
matrix $A$, counted with multiplicity and ordered nonincreasingly, the processes $\bigl(\mathrm{spec}(C_n)\bigr)_{n\ge0}$ and
$\bigl(\mathrm{spec}(\widetilde C_n)\bigr)_{n\ge0}$ have the same law.
\end{lemma}

\begin{proof}
Put $Q_n=S_n^\dagger PS_n$. Then $Q_n=I-u_nu_n^\dagger$ with
$u_n=S_n^\dagger e$ Haar-uniform on the unit sphere, and the $Q_n$ are
independent. From $C_{n}=PS_{n}C_{n-1}S_{n}^\dagger P$ we get the exact
identity
\[
  S_n^\dagger C_nS_n
  =
  \bigl(S_n^\dagger PS_n\bigr)C_{n-1}\bigl(S_n^\dagger PS_n\bigr)
  =
  Q_nC_{n-1}Q_n ,
\]
so $C_n$ and $Q_nC_{n-1}Q_n$ are unitarily conjugate and in particular have
identical spectra. Since $Q_n$ is invariant in law under unitary conjugation
and independent of $C_{n-1}$, the conditional law of
$\mathrm{spec}(Q_nC_{n-1}Q_n)$ given $C_{n-1}$ depends only on
$\mathrm{spec}(C_{n-1})$. The chain $\widetilde C$ has the same transition
kernel on spectra, and $\mathrm{spec}(C_0)=\mathrm{spec}(\widetilde C_0)$. The
claim follows by induction.
\end{proof}

In view of Lemma~\ref{lem:reduction} we work with the chain
$C_0=I$, $C_n=\Pi_nC_{n-1}\Pi_n$ for the rest of this section, writing
\[
  \tr A=\frac1N\Tr A,
  \qquad
  x_p(n)=\tr C_n^p\in[0,1],
  \qquad
  \bar x_p(n)=\E\,x_p(n),
\]
and letting $\mathcal F_n=\sigma(u_1,\ldots,u_n)$ be the sigma-algebra
generated by the first $n$ vectors. Write $x(n)=(x_1(n),x_2(n),\ldots)$ and,
for any indexed family $a$, use $a_{<p}=(a_1,\ldots,a_{p-1})$. Note $0\le C_n\le I$ for
all $n$, since each $\Pi_n$ is a contraction.

\subsection{The exact one-step expansion}\label{subsec:expansion}

\begin{lemma}[Cyclic-cut expansion]\label{lem:expansion}
Let $C$ be any $N\times N$ matrix, $u$ a unit vector, $\Pi=I-uu^\dagger$, and
set $\phi_g=u^\dagger C^gu$ for integers $g\ge1$. Then for every $p\ge1$,
\[
  \Tr\bigl((\Pi C)^p\bigr)
  =
  \Tr(C^p)
  +
  \sum_{k=1}^{p}\frac{(-1)^k p}{k}
  \sum_{\substack{g_1+\cdots+g_k=p\\ g_i\ge1}}
  \;\prod_{i=1}^{k}\phi_{g_i}.
\]
\end{lemma}

\begin{proof}
By cyclicity of the trace,
\[
  \Tr\bigl((\Pi C)^p\bigr)=\Tr\bigl((C\Pi)^p\bigr).
\]
Expanding each of the $p$ factors $\Pi=I-uu^\dagger$ in
$\Tr\bigl(\prod_{j=1}^pC\Pi\bigr)$ gives $2^p$ terms, indexed by the set
$S\subseteq\Z/p\Z$ of positions at which $-uu^\dagger$ is chosen. The term
$S=\varnothing$ is $\Tr(C^p)$. For $|S|=k\ge1$, rank-one-ness of $uu^\dagger$
factorizes the trace. If the cyclic gaps determined by $S$ are
$g_1,\ldots,g_k$ then
\[
  \Tr\bigl(C^{g_1}uu^\dagger C^{g_2}uu^\dagger\cdots C^{g_k}uu^\dagger\bigr)
  =\prod_{i=1}^k u^\dagger C^{g_i}u=\prod_{i=1}^k\phi_{g_i},
\]
so the term equals $(-1)^k\prod_i\phi_{g_i}$, a quantity symmetric in the
$g_i$.

It remains to count. Consider pairs $(S,s)$ with $|S|=k$ and a distinguished
$s\in S$. Reading the gaps of $S$ starting from $s$ gives an ordered
composition $(g_1,\ldots,g_k)$ of $p$ into $k$ positive parts together with the
position $s\in\Z/p\Z$. Conversely, such data determine $(S,s)$.
In the sum over marked pairs, each set $S$ is counted $k$ times, once for
each choice of $s\in S$. On the other hand, each ordered composition is
counted $p$ times, once for each possible starting position $s$.
The weight $\prod_i\phi_{g_i}$ is unchanged by the choice of starting point.
Equating these two ways of summing the weights and dividing by $k$ gives
the total contribution of $|S|=k$ as
\[
  (-1)^k\frac pk
  \sum_{\substack{g_1+\cdots+g_k=p\\g_i\ge1}}\prod_{i=1}^k\phi_{g_i}.
\]
\end{proof}

For $p=3$ this reads $\Tr((\Pi C)^3)=\Tr C^3-3\phi_3+3\phi_1\phi_2-\phi_1^3$.

\begin{lemma}[Sphere moments]\label{lem:sphere}
Let $u$ be Haar-uniform on the unit sphere of $\C^N$. For any matrices
$A_1,\ldots,A_k$,
\[
  \E\prod_{i=1}^{k}u^\dagger A_iu
  =
  \frac{1}{N(N+1)\cdots(N+k-1)}
  \sum_{\sigma\in\mathfrak S_k}\Tr_\sigma(A_1,\ldots,A_k),
\]
where $\mathfrak S_k$ is the permutation group on $\{1,\ldots,k\}$ and
$\Tr_\sigma$ is the product, over the cycles of $\sigma$, of the traces
of the corresponding ordered products of the $A_i$. In particular
$\E[u^\dagger Au]=\tr A$ exactly, and if $\|A_i\|\le1$ for all $i$ then
\begin{equation}\label{eq:sphere-err}
  \left|\E\prod_{i=1}^{k}u^\dagger A_iu-\prod_{i=1}^k\tr A_i\right|
  \le
  \frac{\gamma_k}{N}
\end{equation}
for a constant $\gamma_k$ depending only on $k$.
\end{lemma}

\begin{proof}
Let $\xi=(\xi_1,\ldots,\xi_N)$ be a standard complex Gaussian vector, with
$\xi_i=(X_i+\mathrm{i}Y_i)/\sqrt2$, where all $X_i$ and $Y_i$ are independent
real Gaussian variables with mean zero and variance one. Thus
$\E[\xi_i\overline{\xi_j}]=\delta_{ij}$ and $\E[\xi_i\xi_j]=0$.
The vector $\xi$ is not normalized. Its expected squared norm is $\E\|\xi\|^2=N$.
Then $u=\xi/\|\xi\|$ is
Haar-uniform on the sphere and is independent of $\|\xi\|$, whence
\begin{equation}\label{eq:gauss-split}
  \E\prod_{i=1}^k \xi^\dagger A_i \xi
  =
  \E\bigl[\|\xi\|^{2k}\bigr]\cdot\E\prod_{i=1}^ku^\dagger A_iu .
\end{equation}
Writing $\xi^\dagger A \xi=\sum_{a,b}\bar \xi_aA_{ab}\xi_b$ and applying the Wick
(Isserlis) formula for complex Gaussians, the expectation on the left is a sum
over bijections between the $k$ conjugated and the $k$ unconjugated factors,
i.e.\ over $\sigma\in\mathfrak S_k$, and the contraction indexed by $\sigma$ produces
exactly the product of traces along the cycles of $\sigma$. Thus
\[
  \E\prod_{i=1}^k \xi^\dagger A_i \xi=\sum_{\sigma\in\mathfrak S_k}\Tr_\sigma(A_1,\ldots,A_k).
\]
(For $k=1,2$ this reads $\Tr A$ and $\Tr A\Tr B+\Tr(AB)$.) This is the
vector case of the Weingarten calculus of \cite{CollinsSniady2006}. Since
$\|\xi\|^2$ has the Gamma$(N,1)$ distribution,
$\E\|\xi\|^{2k}=N(N+1)\cdots(N+k-1)$, and
\eqref{eq:gauss-split} gives the stated formula.

For the error bound, the identity permutation contributes
$\prod_i\Tr A_i=N^k\prod_i\tr A_i$. A permutation with $j<k$ cycles
contributes a product of $j$ traces of operators of norm at most $1$, hence at
most $N^j\le N^{k-1}$ in absolute value. Since
$N(N+1)\cdots(N+k-1)=N^k\bigl(1+O_k(N^{-1})\bigr)$ and $|\mathfrak S_k|=k!$, the total
error is $O_k(N^{-1})$.
\end{proof}

\begin{remark}\label{rem:sphere-sharp}
The bound \eqref{eq:sphere-err} is of the right order and cannot be improved to
$O(N^{-2})$. For $k=2$ the exact computation gives
\[
  \E\bigl[(u^\dagger Au)(u^\dagger Bu)\bigr]-\tr A\,\tr B
  =
  \frac{\tr(AB)-\tr A\,\tr B}{N+1},
\]
which for $A=B$ a projection of rank $\lfloor N/2\rfloor$ equals
$\tfrac14(N+1)^{-1}+O(N^{-2})$.
\end{remark}

Define, for $p\ge1$, the polynomial
\begin{equation}\label{eq:vector-field}
  V_p(x)
  =
  -p\,x_p
  +
  H_p(x_1,\ldots,x_{p-1}),
  \qquad
  H_p(x)=\sum_{k=2}^{p}\frac{(-1)^kp}{k}
  \sum_{\substack{g_1+\cdots+g_k=p\\ g_i\ge1}}\prod_{i=1}^kx_{g_i},
\end{equation}
with the convention $H_1=0$ for the empty sum. In $H_p$ every index satisfies $g_i\le p-1$, so $H_p$ involves
only $x_1,\ldots,x_{p-1}$. Thus $V_p$ depends only on $x_1,\ldots,x_p$,
and is linear in $x_p$. This gives the limiting ODE introduced below a
triangular structure. Explicitly,
\[
  V_1=-x_1,\qquad V_2=-2x_2+x_1^2,\qquad V_3=-3x_3+3x_1x_2-x_1^3 .
\]

\begin{lemma}[Drift and increments]\label{lem:drift}
For every integer $p\ge1$, there exists a constant $\eta_p\ge0$, depending
only on $p$, such that the following bounds hold for every $N$, every step
$n$, and every possible current matrix $C_n$.
\begin{equation}\label{eq:drift-eq}
  \E\bigl[x_p(n+1)-x_p(n)\,\big|\,\mathcal F_n\bigr]
  =
  \frac1NV_p\bigl(x(n)\bigr)+R_p(n),
  \qquad
  |R_p(n)|\le\frac{\eta_p}{N^{2}}.
\end{equation}
The increments also satisfy, almost surely,
\[
  \bigl|x_p(n+1)-x_p(n)\bigr|\le\frac{2^p-1}{N}.
\]
\end{lemma}

\begin{proof}
Recall that $C_{n+1}=\Pi_{n+1}C_n\Pi_{n+1}$ and
$x_p(n)=N^{-1}\Tr(C_n^p)$. Since $\Pi_{n+1}^2=\Pi_{n+1}$,
\[
  (\Pi_{n+1}C_n\Pi_{n+1})^p
  =(\Pi_{n+1}C_n)^p\Pi_{n+1}.
\]
Taking traces and using cyclicity to move the final projection to the front,
where it is absorbed by the first projection, gives
\[
  \Tr(C_{n+1}^p)=\Tr\bigl((\Pi_{n+1}C_n)^p\bigr).
\]
Thus the moment increment can be written in the form required by
Lemma~\ref{lem:expansion}. Namely,
\[
  N\bigl(x_p(n+1)-x_p(n)\bigr)
  =\Tr\bigl((\Pi_{n+1}C_n)^p\bigr)-\Tr(C_n^p).
\]
Apply that lemma with $C=C_n$, $u=u_{n+1}$, and
$\Pi=\Pi_{n+1}=I-u_{n+1}u_{n+1}^\dagger$. Writing
$\phi_g=u_{n+1}^\dagger C_n^g u_{n+1}$ for $g\ge1$, we obtain
\[
  N\bigl(x_p(n+1)-x_p(n)\bigr)
  =\sum_{k=1}^{p}\frac{(-1)^kp}{k}
  \sum_{\substack{g_1+\cdots+g_k=p\\g_i\ge1}}
  \prod_{i=1}^k\phi_{g_i}.
\]
%The term $\Tr(C_n^p)$ from the lemma has canceled in the difference.

Since $0\le C_n\le I$, we also have $0\le C_n^g\le I$, and hence
$0\le\phi_g\le1$ because $u_{n+1}$ is a unit vector. Every product in the
inner sum therefore has absolute value at most one. For a fixed $k$, there
are $\binom{p-1}{k-1}$ ordered compositions of $p$ into $k$ positive parts. To see this, choose $k-1$ cuts among the $p-1$ gaps between $p$ consecutive units.
The triangle inequality consequently gives
\[
  N\bigl|x_p(n+1)-x_p(n)\bigr|
  \le\sum_{k=1}^p\frac pk\binom{p-1}{k-1}
  =\sum_{k=1}^p\binom pk
  =2^p-1.
\]
Dividing by $N$ proves the increment bound.

For the drift, conditioned on $\mathcal F_n$, the matrix $C_n$ is then known,
while $u_{n+1}$ remains independent and uniform on the sphere.
Lemma~\ref{lem:sphere}, applied with $A_i=C_n^{g_i}$, gives
\[
  \E\left[\prod_{i=1}^k\phi_{g_i}\,\middle|\,\mathcal F_n\right]
  =\prod_{i=1}^k x_{g_i}(n)+O_p(N^{-1}),
\]
uniformly in $C_n$, by \eqref{eq:sphere-err}. For $k=1$ this equality is
exact, and its contribution to the drift is $-p x_p(n)$. The terms with
$k\ge2$ give $H_p(x_1(n),\ldots,x_{p-1}(n))$ from
\eqref{eq:vector-field}. Summing the errors with the coefficients $p/k$
therefore yields
\[
  \E\bigl[N(x_p(n+1)-x_p(n))\mid\mathcal F_n\bigr]
  =V_p(x(n))+O_p(N^{-1}).
\]
Dividing by $N$ gives \eqref{eq:drift-eq}, with a remainder bounded by
$\eta_p/N^2$ for a constant $\eta_p$ depending only on $p$.
\end{proof}

\subsection{Concentration}\label{subsec:conc}

Trace-norm contraction controls the effect of changing one projection.

\begin{lemma}[Trace-norm contraction]\label{lem:contract}
For every matrix $D$ and every orthogonal projection $\Pi$,
\[
  \|\Pi D\Pi\|_1\le\|D\|_1.
\]
\end{lemma}

\begin{proof}
By trace-norm/operator-norm duality,
\[
  \|\Pi D\Pi\|_1
  =\sup_{\|Y\|\le1}\bigl|\Tr(Y\Pi D\Pi)\bigr|
  =\sup_{\|Y\|\le1}\bigl|\Tr\bigl((\Pi Y\Pi)D\bigr)\bigr|
  \le\sup_{\|Z\|\le1}\bigl|\Tr(ZD)\bigr|
  =\|D\|_1,
\]
where the middle equality is cyclicity of the trace together with
$\Pi^2=\Pi$, and the inequality uses $\|\Pi Y\Pi\|\le\|Y\|\le1$.
\end{proof}

\begin{lemma}[Single-vector sensitivity]\label{lem:lip}
Fix $p\ge1$ and $L\ge1$, and regard
$\Phi(u_1,\ldots,u_L)=\tr\bigl(C_L^p\bigr)$ as a function of the $L$ unit
vectors. Then for every $i\le L$ and all unit vectors $u_i,u_i'$, with the
other arguments held fixed,
\[
  \bigl|\Phi(\ldots,u_i,\ldots)-\Phi(\ldots,u_i',\ldots)\bigr|
  \le
  \frac{8\,p}{N}\,\|u_i-u_i'\| .
\]
\end{lemma}

\begin{proof}
Let $C_n$ and $C_n'$ denote the two chains. They agree for $n<i$. Write
$D=C_{i-1}=C_{i-1}'$, $R=u_iu_i^\dagger$, $R'=u_i'u_i'^\dagger$ and
$\Delta=R-R'$. Set $\Pi=I-R$ and $\Pi'=I-R'$. From
\[
  \Pi D\Pi=D-RD-DR+RDR,
\]
and the corresponding identity for $\Pi'$, we get
\[
  C_i-C_i'=-\Delta D-D\Delta+\bigl(RDR-R'DR'\bigr)
  =-\Delta D-D\Delta+RD\Delta+\Delta DR' .
\]
The trace norm of $\Delta$ can be computed exactly. Since
$\Delta=u_iu_i^\dagger-u_i'u_i'^\dagger$ is Hermitian of rank at most $2$ with
$\Tr\Delta=0$ and
\[
  \Tr\Delta^2=2\bigl(1-|\langle u_i,u_i'\rangle|^2\bigr)=2\sin^2\theta,
  \qquad \theta:=\arccos|\langle u_i,u_i'\rangle|\in[0,\pi/2],
\]
its nonzero eigenvalues are exactly $\pm\sin\theta$. Hence
$\|\Delta\|=\sin\theta$ and $\|\Delta\|_1=2\sin\theta$. On the other hand
$\|u_i-u_i'\|\ge2\sin(\theta/2)$, since
$\|u_i-u_i'\|^2=2-2\operatorname{Re}\langle u_i,u_i'\rangle\ge2-2\cos\theta
=4\sin^2(\theta/2)$. Since $2\sin\theta=2\cos(\theta/2)\cdot2\sin(\theta/2)$,
\begin{equation}\label{eq:delta-sharp}
  \|\Delta\|_1
  =
  2\cos(\theta/2)\cdot2\sin(\theta/2)
  \le
  2\cos(\theta/2)\,\|u_i-u_i'\|
  \le
  2\,\|u_i-u_i'\| ,
\end{equation}
and the constant $2$ is sharp, being attained in the limit $u_i'\to u_i$. Using
$\|XY\|_1\le\|X\|\,\|Y\|_1$ and $\|XY\|_1\le\|X\|_1\|Y\|$ together with
$\|R\|=\|R'\|=1$ and $\|D\|\le1$, each of the four terms above has trace norm
at most $\|\Delta\|_1$, so
\[
  \|C_i-C_i'\|_1\le4\|\Delta\|_1\le8\,\|u_i-u_i'\| .
\]
For $n>i$ the two chains use the same $\Pi_n$, so
$C_n-C_n'=\Pi_n(C_{n-1}-C_{n-1}')\Pi_n$ and Lemma~\ref{lem:contract} gives
$\|C_n-C_n'\|_1\le\|C_{n-1}-C_{n-1}'\|_1$. Iterating,
\[
  \|C_L-C_L'\|_1\le8\,\|u_i-u_i'\|
\]
uniformly in $L-i$. Finally, for $0\le A,B\le I$,
\[
  A^p-B^p=\sum_{j=0}^{p-1}A^{j}(A-B)B^{p-1-j}
  \quad\Longrightarrow\quad
  \|A^p-B^p\|_1\le p\,\|A-B\|_1 ,
\]
so $|\Tr(C_L^p)-\Tr(C_L'^p)|\le p\|C_L-C_L'\|_1$, and dividing by $N$ gives the
claim.
\end{proof}

\begin{proposition}[Variance bound]\label{prop:conc}
For every $p\ge1$ and $T>0$ there is $\Lambda^{\mathrm{var}}_{p,T}$ such that for all $N\ge1$ and
all $L\le TN$,
\[
  \Var\bigl(\tr C_L^p\bigr)\le\frac{\Lambda^{\mathrm{var}}_{p,T}}{N^{2}} .
\]
One may take $\Lambda^{\mathrm{var}}_{p,T}=64\,p^2T$.  
\end{proposition}

\begin{proof}
We use tensorization of the Poincar\'e inequality.

\emph{Step 1. Poincar\'e on one sphere.} The uniform probability measure
$\sigma$ on the unit sphere $S^{d-1}\subset\R^{d}$ satisfies
\[
  \Var_\sigma(f)\le\frac{1}{d-1}\int_{S^{d-1}}|\nabla_{S}f|^2\,d\sigma
\]
for all locally Lipschitz $f$, where $\nabla_S$ is the Riemannian gradient. The constant is the reciprocal of the spectral gap of the Laplace--Beltrami
operator on $S^{d-1}$, whose first nonzero eigenvalue is $d-1$
\cite[Ch.~2]{Ledoux2001}, \cite[\S4.4]{BakryGentilLedoux2014}. The unit sphere
of $\C^N$ is $S^{2N-1}\subset\R^{2N}$, so its Poincar\'e constant is
$(2N-1)^{-1}$.

\emph{Step 2. Tensorization.} If each factor of a product measure satisfies a
Poincar\'e inequality with the same constant $c_{\mathrm P}$, then the product measure $\sigma^{\otimes L}$ satisfies
\[
  \Var_{\sigma^{\otimes L}}(\Phi)\;\le\;c_{\mathrm P}\int\sum_{i=1}^{L}\bigl|\nabla_{u_i}\Phi\bigr|^2\,d\sigma^{\otimes L},
\]
where $\nabla_{u_i}$ is the Riemannian gradient in the $i$-th variable with the
others held fixed. This is the standard tensorization property of the spectral
gap for product spaces \cite[\S4.4]{BakryGentilLedoux2014}. Here
$c_{\mathrm P}=(2N-1)^{-1}$.

\emph{Step 3. The gradient bound.} Fix $i$ and fix all $u_j$, $j\ne i$. By
Lemma~\ref{lem:lip}, the map $u_i\mapsto\Phi$ satisfies
\[
  |\Phi(\ldots,u_i,\ldots)-\Phi(\ldots,u_i',\ldots)|
  \le\frac{8p}{N}\,\|u_i-u_i'\|_{\R^{2N}} .
\]
For any two points of the sphere the Euclidean chord is at most the geodesic
arc, $\|u_i-u_i'\|_{\R^{2N}}\le d_{S}(u_i,u_i')$. Hence the same inequality
holds with $d_S$ on the right, i.e.\ $\Phi$ is $8p/N$-Lipschitz for the
geodesic distance. Since $\Phi$ is a polynomial in the real coordinates of
the $u_j$, it is smooth, and this Lipschitz bound gives
$|\nabla_{u_i}\Phi|\le8p/N$ everywhere.

\emph{Step 4. Conclusion.} Combining Steps 2 and 3 and using $L\le TN$,
\[
  \Var(\Phi)
  \;\le\;\frac{1}{2N-1}\cdot L\cdot\left(\frac{8p}{N}\right)^{2}
  \;\le\;\frac{TN}{N}\cdot\frac{64\,p^{2}}{N^{2}}
  \;=\;\frac{64\,p^{2}T}{N^{2}},
\]
where we used $2N-1\ge N$ for $N\ge1$.
\end{proof}

\begin{corollary}[Decorrelation]\label{cor:decorr}
Let $r\ge1$ and integers $a_1,\ldots,a_r\ge1$ be fixed, let $T>0$, and let
$L\le TN$. Then
\[
  \left|\E\prod_{j=1}^r\tr\bigl(C_L^{a_j}\bigr)
  -\prod_{j=1}^r\E\,\tr\bigl(C_L^{a_j}\bigr)\right|
  \le\frac{\eta'}{N^{2}}
\]
with $\eta'=\eta'(r,a_1,\ldots,a_r,T)$.
\end{corollary}

\begin{proof}
Write $X_j=\tr(C_L^{a_j})\in[0,1]$ and
$v=\max_j\Var(X_j)\le N^{-2}\max_j\Lambda^{\mathrm{var}}_{a_j,T}$ by
Proposition~\ref{prop:conc}. Replacing the factors by their expectations one
at a time gives the telescoping identity
\[
  \prod_jX_j-\prod_j\E X_j
  =\sum_{j=1}^r\Bigl(\prod_{l<j}X_l\Bigr)(X_j-\E X_j)
  \Bigl(\prod_{l>j}\E X_l\Bigr),
\]
so $\bigl\|\prod_jX_j-\prod_j\E X_j\bigr\|_{L^2}\le r\sqrt v$. Applying this
with $r-1$ factors bounds $\Var\bigl(\prod_{l\ge2}X_l\bigr)\le(r-1)^2v$.
Telescoping and using Cauchy--Schwarz at each step,
\[
  \left|\E\prod_jX_j-\prod_j\E X_j\right|
  \le\sum_{j=1}^{r-1}\sqrt{v\cdot(r-j)^2v}
  \le\frac{r^2}{2}\,v ,
\]
which is $O(N^{-2})$.
\end{proof}

\subsection{The limiting flow and accumulation}\label{subsec:accum}

\begin{lemma}\label{lem:ode}
For each $K\ge1$ the initial value problem
\[
  \dot y_p=V_p(y),\qquad y_p(0)=1,\qquad p=1,\ldots,K,
\]
has the unique solution $y_p(t)=m_p(t)$, the $p$-th moment of $\mu_t$.
\end{lemma}

\begin{proof}
By \eqref{eq:vector-field} the system is triangular. Its $p$-th equation is
$\dot y_p=-py_p+H_p(y_{<p})$. Once $y_1,\ldots,y_{p-1}$ are known,
$H_p(y_{<p}(t))$ is a known function of time, so this is a linear equation
for $y_p$. Induction on $p$ gives global existence and uniqueness on
$[0,\infty)$ and extends the solution uniquely to all $p\ge1$.
We regard $(y_p)_{p\ge1}$ as defined for every $p$ and set
$X(z,t)=\sum_{p\ge1}y_p(t)z^p$. Since
$X$ has no constant term, $[z^p]X^k=\sum_{g_1+\cdots+g_k=p}\prod_iy_{g_i}$ with
all $g_i\ge1$, so only $k\le p$ contributes. Also,
$\sum_{k\ge1}\frac{(-1)^k}{k}X^k=-\log(1+X)$. Together with
$\sum_pp\,a_pz^p=z\partial_z\sum_pa_pz^p$, the system is therefore equivalent
to
\begin{equation}\label{eq:pde}
  \partial_tX
  =
  -z\,\partial_z\log(1+X)
  =
  -\frac{z\,\partial_zX}{1+X},
  \qquad
  X(z,0)=\frac{z}{1-z},
\end{equation}
the initial condition being $\sum_{p\ge1}z^p$, which is the moment series of
$\mu_0=\delta_1$.

It suffices to check that $X=\psi_t(z)$, the moment series of $\mu_t$, solves
\eqref{eq:pde}. By Theorem~\ref{thm:free-small-loss},
$\psi_t^{-1}(w)=\frac{w}{1+w}\exp\bigl(\frac{t}{1+w}\bigr)$, so with
$X=\psi_t(z)$ the relation
\[
  z=\Psi(X,t):=\frac{X}{1+X}\exp\left(\frac{t}{1+X}\right)
\]
holds identically for $z$ in a neighborhood of $0$. Differentiating in $z$ at
fixed $t$ gives $1=\Psi_X\,\partial_zX$, so in particular $\Psi_X\ne0$ there.
Differentiating in $t$ at fixed $z$ gives $0=\Psi_X\,\partial_tX+\Psi_t$. Since
\[
  \Psi_t=\frac{X}{1+X}\exp\left(\frac{t}{1+X}\right)\cdot\frac1{1+X}
  =\frac{z}{1+X},
\]
we obtain $\partial_tX=-\Psi_t/\Psi_X=-\frac{z}{1+X}\partial_zX$, which is
\eqref{eq:pde}. At $t=0$ the relation $z=X/(1+X)$ inverts to $X=z/(1-z)$, so
the initial condition holds. Comparing coefficients of $z^p$ for $p\le K$ and
invoking uniqueness for the finite triangular polynomial system gives
$y_p=m_p$.
\end{proof}

\begin{remark}
Equation~\eqref{eq:pde} can be written as
\[
  \partial_tX+\frac{z}{1+X}\,\partial_zX=0.
\]
This is an inviscid transport equation. The value of $X$ is constant along characteristic
curves, and there is no second-order diffusion term. Along such a curve,
$dz/dt=z/(1+X)$, so $z(t)=z(0)\exp(t/(1+X))$. The initial condition gives
$z(0)=X/(1+X)$, and hence
\[
  z=\frac{X}{1+X}\exp\left(\frac{t}{1+X}\right).
\]
This recovers the inverse moment-series relation and explains the exponential
factor in the $S$-transform.
\end{remark}

\begin{remark}
The condition $\Psi_X\ne0$ near $z=0$ follows from the inverse function
relation. The vanishing associated with the known branch
point occurs away from the formal-power-series regime $z\to0$ used here.
\end{remark}

\begin{proof}[Proof of Theorem~\ref{thm:double-scaling}]
Fix $K\ge1$ and $T>\tau$, and set
$\delta_p(n)=\bar x_p(n)-m_p(n/N)$ for $p\le K$ and $n\le TN$.
Throughout this proof, the constants $b_p$ and $A_p$ and the implicit
constants in $O_p$ may also depend on the fixed time horizon $T$.
They remain independent of $N$ and $n$.

\emph{Step 1. A damped recursion for the bias.} Taking expectations in
\eqref{eq:drift-eq} and using $V_p(x)=-px_p+H_p(x_{<p})$,
\[
  \bar x_p(n+1)
  =\Bigl(1-\frac pN\Bigr)\bar x_p(n)
  +\frac1N\,\E\bigl[H_p\bigl(x_{<p}(n)\bigr)\bigr]
  +\E\bigl[R_p(n)\bigr],
  \qquad
  \bigl|\E R_p(n)\bigr|\le\frac{\eta_p}{N^{2}} .
\]
Since $H_p$ is a polynomial in the variables $x_1,\ldots,x_{p-1}$, all taking
values in $[0,1]$, Corollary~\ref{cor:decorr} applied monomial by monomial
gives $\E[H_p(x_{<p}(n))]=H_p(\bar x_{<p}(n))+O_p(N^{-2})$. On the other hand,
by triangularity each $m_p$ solves a linear inhomogeneous ODE driven by
$m_1,\ldots,m_{p-1}$ and takes values in $[0,1]$, being a moment of a
probability measure on $[0,1]$.  Thus $m_p'=V_p(m)$ and
$m_p''=\nabla V_p(m)\cdot V(m)$ are bounded on $[0,T+1]$, and $m_p\in
C^2([0,T+1])$ (the interval is enlarged so that the point $(n+1)/N$ lies in it
for every $n\le TN$). Taylor expansion and Lemma~\ref{lem:ode} then give
\[
  m_p\Bigl(\frac{n+1}{N}\Bigr)
  =\Bigl(1-\frac pN\Bigr)m_p\Bigl(\frac nN\Bigr)
  +\frac1NH_p\Bigl(m_{<p}\Bigl(\frac nN\Bigr)\Bigr)
  +O_p(N^{-2}).
\]
Subtracting and writing $J_p$ for the Lipschitz constant of $H_p$ on
$[0,1]^{p-1}$,
\begin{equation}\label{eq:bias-rec}
  |\delta_p(n+1)|
  \le
  \Bigl(1-\frac pN\Bigr)|\delta_p(n)|
  +\frac{J_p}{N}\max_{q<p}|\delta_q(n)|
  +\frac{b_p}{N^{2}}
\end{equation}
for a constant $b_p$ and all $N\ge p$, with the maximum interpreted as zero
when $p=1$.

\emph{Step 2. Induction on $p$.} We claim $|\delta_p(n)|\le A_p/N$ for all
$n\le TN$, with $A_p$ independent of $N$ and $n$. For $p=1$ we have $H_1=0$
and $\delta_1(0)=0$, so iterating \eqref{eq:bias-rec},
\[
  |\delta_1(n)|\le\frac{b_1}{N^{2}}\sum_{j\ge0}\Bigl(1-\frac1N\Bigr)^{j}
  =\frac{b_1}{N^{2}}\cdot N=\frac{b_1}{N},
\]
so $A_1=b_1$. Assuming the claim for all $q<p$ with constants $A_q$, put
$A_{p-1}^{*}=\max_{q<p}A_q$. Then \eqref{eq:bias-rec} gives
\[
  |\delta_p(n+1)|\le\Bigl(1-\frac pN\Bigr)|\delta_p(n)|
  +\frac{J_pA^{*}_{p-1}+b_p}{N^{2}},
\]
and iterating from $\delta_p(0)=0$,
\[
  |\delta_p(n)|
  \le\frac{J_pA^{*}_{p-1}+b_p}{N^{2}}\sum_{j\ge0}\Bigl(1-\frac pN\Bigr)^{j}
  =\frac{J_pA^{*}_{p-1}+b_p}{p\,N}.
\]
This proves the claim with $A_p=(J_pA^{*}_{p-1}+b_p)/p$, for $N\ge p$. For the
finitely many $N<p$ the bound is trivial, since $x_p(n)$ and $m_p$ both lie in
$[0,1]$ and hence $|\delta_p(n)|\le1\le p/N$. Enlarging $A_p$ to
$\max\{A_p,p\}$ covers all $N$.

\emph{Step 3. From bias to convergence.} Set
$\Lambda_{p,T}=\max\{A_p,\Lambda^{\mathrm{var}}_{p,T}\}$, so that a single
constant bounds both the bias from Step 2 and the variance from
Proposition~\ref{prop:conc}. Then, for all $N\ge1$ and integers $0\le n\le TN$,
\begin{equation}\label{eq:quantitative-bounds}
  \bigl|\E\,x_p(n)-m_p(n/N)\bigr|\le\frac{\Lambda_{p,T}}{N},
  \qquad
  \Var\bigl(x_p(n)\bigr)\le\frac{\Lambda_{p,T}}{N^2}.
\end{equation}
In particular,
\[
  \E\bigl|x_p(n)-m_p(n/N)\bigr|^2
  =\Var\bigl(x_p(n)\bigr)+\delta_p(n)^2
  \le\frac{\Lambda_{p,T}+\Lambda_{p,T}^2}{N^{2}} .
\]
Hence $\|x_p(n)-m_p(n/N)\|_{L^2}=O(N^{-1})$, which also gives convergence
in probability by Chebyshev's inequality. We now take $n=L_N$.
By \eqref{eq:double-scaling-regime}, $L_N\le TN$ for all sufficiently large
$N$. The triangle inequality gives
\[
  \|x_p(L_N)-m_p(\tau)\|_{L^2}
  \le \|x_p(L_N)-m_p(L_N/N)\|_{L^2}
     +|m_p(L_N/N)-m_p(\tau)|.
\]
The first term tends to zero by the preceding estimate. The second tends to
zero by continuity of $m_p$ and \eqref{eq:double-scaling-regime}.
Thus $x_p(L_N)\to m_p(\tau)$ in $L^2$, and hence in probability.
This proves the first assertion of the theorem.

\emph{Step 4. Weak convergence.} Let $f\in C([0,1])$ and $\varepsilon>0$. By
the Weierstrass approximation theorem there is a polynomial $g$, of some
finite degree $K$, with $\|f-g\|_\infty<\varepsilon/3$. For any probability
measures $\nu,\mu$ on $[0,1]$,
\[
  \left|\int f\,d\nu-\int f\,d\mu\right|
  \le\frac{2\varepsilon}{3}+\left|\int g\,d\nu-\int g\,d\mu\right|,
\]
and the last term involves only moments of order at most $K$. Applying Step 3
with this $K$ shows
$\Pr\bigl(\bigl|\int f\,d\ESD(C_{L_N})-\int f\,d\mu_\tau\bigr|>\varepsilon\bigr)\to0$.
Since $f$ was arbitrary, $\ESD(C_{L_N})\to\mu_\tau$ weakly, in probability. By
Lemma~\ref{lem:reduction} the same holds for $\ESD(B_{L_N}^\dagger B_{L_N})$.

\emph{Step 5. Factorization.} For a fixed composition
$\alpha=(\alpha_1,\ldots,\alpha_r)$, Corollary~\ref{cor:decorr} and Step 2 give
\[
  \E\prod_{j=1}^r\tr\bigl(C_{L_N}^{\alpha_j}\bigr)
  =\prod_{j=1}^r\E\,\tr\bigl(C_{L_N}^{\alpha_j}\bigr)+O(N^{-2})
  =\prod_{j=1}^rm_{\alpha_j}\!\left(\frac{L_N}{N}\right)+O(N^{-1}).
\]
By \eqref{eq:double-scaling-regime} and continuity of each $m_{\alpha_j}$, the right-hand side
tends to $\prod_jm_{\alpha_j}(\tau)$, which is the asymptotic factorization of
normalized trace products. The error is $O(N^{-1})$ relative to
$\prod_jm_{\alpha_j}(L_N/N)$, but only $o(1)$ relative to
$\prod_jm_{\alpha_j}(\tau)$. The hypothesis \eqref{eq:double-scaling-regime} carries no rate, so
$|m_{\alpha_j}(L_N/N)-m_{\alpha_j}(\tau)|$ cannot be quantified. If one assumes
in addition $L_N/N=\tau+O(N^{-1})$, then the rate $O(N^{-1})$ holds relative to
$\prod_jm_{\alpha_j}(\tau)$ as well. The same remark applies to Step 3.
\end{proof}

\subsection{The orthogonal ensemble}\label{subsec:orthogonal}

The computation in \cite{Beenakker2501} uses the real (orthogonal)
model
\[
  Q_L=\Pi_L\cdots\Pi_1,
  \qquad \Pi_n=I-u^{(n)}(u^{(n)})^\top,
\]
where the $u^{(n)}$ are independent uniform vectors on the real unit sphere.
Here $\top$ denotes the transpose.
We next establish the double-scaling limit for the orthogonal ensemble.

Concretely, let $P=I-ee^\top$ with $e\in\R^N$ a unit vector, let
$S_1,S_2,\ldots$ be independent Haar-distributed elements of the orthogonal
group $O(N)$, and put $B_L=(PS_L)\cdots(PS_1)$ and $C_n=B_nB_n^\top$. Then
$R_n=S_n^\top PS_n=I-u_nu_n^\top$ with $u_n=S_n^\top e$ Haar-uniform on the
real unit sphere $S^{N-1}\subset\R^N$, and the $u_n$ are independent.

\begin{theorem}[Double-scaling limit, orthogonal ensemble]\label{thm:double-scaling-orth}
The conclusions of Theorem~\ref{thm:double-scaling} hold for independent Haar
orthogonal matrices, with conjugate transposes replaced by transposes and
the same limiting law $\mu_\tau$.
\end{theorem}

\begin{proof}
The reduction to the projection chain, the algebraic expansion, the trace-norm
contraction, and the accumulation argument are unchanged. Only the sphere
moments and Poincar\'e constant differ. For a Haar vector on $S^{N-1}$,
Lemma~\ref{lem:sphere-real} replaces Lemma~\ref{lem:sphere}. Its error remains
$O(N^{-1})$ for each fixed product of quadratic forms, so the limiting drift
$V_p$ is unchanged. The real sphere-moment constant counts $(2k-1)!!$
matchings in place of the $k!$ permutations in the complex case.

The spectral gap of $S^{N-1}$ is $N-1$ instead of $2N-1$. Thus, for $N\ge2$
and $L\le TN$, the variance calculation gives
\[
  \Var\bigl(\tr C_L^p\bigr)
  \le TN\,\frac{1}{N-1}\left(\frac{8p}{N}\right)^2
  \le \frac{128p^2T}{N^2}.
\]
The displayed variance constant is twice the one in Proposition~\ref{prop:conc}.
The same bias recursion then gives the bounds
\eqref{eq:quantitative-bounds} with a constant $\Lambda'_{p,T}$ depending only
on $p$ and $T$. Its value accounts for both the real sphere-moment estimate
and the changed variance bound. The case $N=1$ is deterministic and is
covered by enlarging this constant if necessary.
The limiting ODE and its solution are unchanged, so the remainder of the
proof of Theorem~\ref{thm:double-scaling} gives all the asserted convergence
and factorization conclusions.
\end{proof}

\begin{lemma}[Real sphere moments]\label{lem:sphere-real}
Let $u$ be Haar-uniform on the unit sphere of $\R^N$ and let
$A_1,\ldots,A_k$ be real symmetric matrices with $\|A_i\|\le1$. Then
\[
  \left|\E\prod_{i=1}^{k}u^\top A_iu-\prod_{i=1}^k\tr A_i\right|
  \le\frac{\gamma_k}{N}
\]
for a constant $\gamma_k$ depending only on $k$, and $\E[u^\top Au]=\tr A$ exactly.
\end{lemma}

\begin{proof}
\emph{Step 1. Pass from the sphere to a Gaussian vector.}
Let $\xi=(\xi_1,\ldots,\xi_N)$ have independent real Gaussian coordinates
of mean zero and variance one. Its direction $u=\xi/\|\xi\|$ is uniform on
the sphere and independent of its length. Each quadratic form satisfies
$\xi^\top A_i\xi=\|\xi\|^2u^\top A_i u$. Multiplying these identities and
using independence gives
\[
  \E\prod_{i=1}^k u^\top A_i u
  =\frac{\E\prod_{i=1}^k\xi^\top A_i\xi}{d_{N,k}},
  \qquad
  d_{N,k}:=\E\|\xi\|^{2k}
  =N(N+2)\cdots(N+2k-2).
\]
The formula for $d_{N,k}$ follows because $\|\xi\|^2$ has the $\chi^2_N$
distribution.

\emph{Step 2. Identify the leading Gaussian contribution.}
Expand the numerator using
\[
  \xi^\top A_i\xi=\sum_{a_i,b_i=1}^N
  (A_i)_{a_i b_i}\xi_{a_i}\xi_{b_i}.
\]
There are $2k$ Gaussian factors, two for each matrix. Wick's formula sums
over all $(2k-1)!!$ ways to pair these factors. A pair contributes the
covariance $\E[\xi_a\xi_b]=\delta_{ab}$.
One pairing joins the two factors belonging to each matrix separately.
Its contribution is
\[
  \prod_{i=1}^k\sum_{a_i=1}^N(A_i)_{a_i a_i}
  =\prod_{i=1}^k\Tr A_i
  =N^k\prod_{i=1}^k\tr A_i.
\]

\emph{Step 3. Bound the other pairings.}
To see the form of a pairing, represent each matrix by a vertex with its two
Gaussian factors attached. The pairs form edges between these vertices.
Every vertex has degree two, with a loop counted twice, so each connected
component is a cycle. The covariance constraints identify the indices along
each cycle. Summing those indices gives the trace of a product of the
matrices on that cycle. Symmetry of the $A_i$ allows any transposed factors
to be written as $A_i$ themselves.
For example, when $k=2$, pairing within each matrix gives
$\Tr A_1\Tr A_2$, while each of the other two pairings gives $\Tr(A_1A_2)$.

The pairing in Step 2 has $k$ separate loops. Every other pairing connects
at least two distinct vertices and therefore has at most $k-1$ cycles.
Each trace has absolute value at most $N$, since every matrix in its product
has operator norm at most one. Consequently,
\[
  \left|\E\prod_{i=1}^k\xi^\top A_i\xi
       -N^k\prod_{i=1}^k\tr A_i\right|
  \le\bigl((2k-1)!!-1\bigr)N^{k-1}.
\]
Since $d_{N,k}\ge N^k$ and $|\tr A_i|\le1$, division by $d_{N,k}$ gives
\[
  \left|\E\prod_{i=1}^k u^\top A_i u-\prod_{i=1}^k\tr A_i\right|
  \le\frac{(2k-1)!!-1}{N}
      +\left|\frac{N^k}{d_{N,k}}-1\right|.
\]
Finally,
\[
  0\le1-\frac{N^k}{d_{N,k}}
  =1-\prod_{j=0}^{k-1}\left(1+\frac{2j}{N}\right)^{-1}
  \le\sum_{j=0}^{k-1}\frac{2j}{N}
  =\frac{k(k-1)}{N}.
\]
This proves the bound with a constant depending only on $k$.
For $k=1$, the numerator is $\Tr A$ and $d_{N,1}=N$, giving
$\E[u^\top Au]=\tr A$ exactly. In our application, $A_i=C_n^{g_i}$ is
symmetric because $C_n$ is real symmetric.
\end{proof}

\begin{remark}
For $k=2$ the real formula reads
\[
  \E\bigl[(u^\top Au)(u^\top Bu)\bigr]
  =\frac{\Tr A\,\Tr B+2\Tr(AB)}{N(N+2)},
\]
the factor $2$ and the shift $N+2$ distinguishing it from the complex case of
Lemma~\ref{lem:sphere}. The discrepancy from $\tr A\,\tr B$ is again of order
$N^{-1}$ and not better, exactly as in Remark~\ref{rem:sphere-sharp}.
\end{remark}

\begin{corollary}[Orthogonal-model moment formula]\label{cor:beenakker-settled}
Let $Q_L=\Pi_L\cdots\Pi_1$, where
$\Pi_n=I-u^{(n)}(u^{(n)})^\top$ and the $u^{(n)}$ are independent uniform
vectors on the real unit sphere. Thus $Q_L$ is a product of independent
Haar-random real rank-$(N-1)$ projectors with $L=L_N$ satisfying
\eqref{eq:double-scaling-regime}. Then for every fixed $p\ge1$,
\[
  \lim_{N\to\infty}\frac1N\,\E\bigl[\Tr(Q_LQ_L^\top)^p\bigr]
  =
  \eur^{-p\tau}\frac{G_p(\tau)}{(p-1)!},
  \qquad
  G_p(\tau)=(1-\tau)(p-1)!\sum_{i=0}^{p-1}\frac{(p\tau)^i}{i!}+\tau^pp^{p-1}.
\]
In particular the closed form conjectured in
\cite[Eq.~(13a)]{Beenakker2501}, and verified there for $p\le 6$, holds for
every $p$. At $\tau=1$ one recovers
$\lim_N\frac1N\E[\Tr(Q_NQ_N^\top)^p]=(p/\eur)^p/p!$.
\end{corollary}

\begin{proof}
\emph{Step 1. Identify the compression chain.}
Set $Q_0=I$ and $Q_n=\Pi_n\cdots\Pi_1$. Since $\Pi_n^\top=\Pi_n$,
\[
  Q_nQ_n^\top
  =(\Pi_nQ_{n-1})(Q_{n-1}^\top\Pi_n)
  =\Pi_n(Q_{n-1}Q_{n-1}^\top)\Pi_n.
\]
Thus $Q_nQ_n^\top$ starts at $I$ and follows exactly the recursion defining
$\widetilde C_n$ in Lemma~\ref{lem:reduction}.

\emph{Step 2. Apply the orthogonal limit.}
The orthogonal version of Lemma~\ref{lem:reduction} identifies the eigenvalue
law of this chain with that of $B_nB_n^\top$. Theorem~\ref{thm:double-scaling-orth}
therefore gives, for each fixed $p\ge1$,
\[
  \frac1N\Tr\bigl((Q_{L_N}Q_{L_N}^\top)^p\bigr)
  \xrightarrow{L^2}m_p(\tau).
\]
Convergence in $L^2$ implies convergence of expectations by
Cauchy--Schwarz's inequality. Hence
\[
  \lim_{N\to\infty}\frac1N\E\Tr\bigl((Q_{L_N}Q_{L_N}^\top)^p\bigr)
  =m_p(\tau).
\]

\emph{Step 3. Insert the moment formula.}
Theorem~\ref{thm:beenakker-closed-form} identifies this limit as
$\eur^{-p\tau}G_p(\tau)/(p-1)!$ with the stated expression for $G_p$.
At $\tau=1$, its first term vanishes, so $G_p(1)=p^{p-1}$. Therefore,
\[
  m_p(1)=\frac{\eur^{-p}p^{p-1}}{(p-1)!}
  =\frac{(p/\eur)^p}{p!}.
\]
\end{proof}

\begin{remark}
The spectral identification in Lemma~\ref{lem:reduction} is an equality in
law. It does not identify the matrices $Q_L$ and $B_L$. Even if
$S_n^\top e=u^{(n)}$, the identity $PS_n=S_n\Pi_n$ gives
\[
  B_L=S_L\Pi_LS_{L-1}\Pi_{L-1}\cdots S_1\Pi_1,
\]
so the orthogonal factors remain between the projections.
\end{remark}

%-----------------------------------------------------------------------
\section*{Acknowledgements}

The author thanks Carlo Beenakker for a short but helpful correspondence about
monitored quantum transport and the moment formulas that motivated this note,
and Lutz Warnke for pointing out the references on Lagrange inversion.

\section*{Use of generative AI}

The author identified the connection between Beenakker's transmission
moments, free multiplicative convolution, and the $S$-transform, and
confirmed this connection through correspondence with him.
The original arXiv version already contained the fixed-length and free
small-loss limits, together with an explicit formula for every moment of
the candidate limiting law, derived by Lagrange inversion.

During the revision, Claude proposed the strategy for establishing the
double-scaling convergence now stated in Theorem~\ref{thm:double-scaling}.
The author developed and proved this result with Claude's assistance,
completing the link between the candidate law and the limiting moments of
the monitored matrix products. ChatGPT assisted with removing redundant
arguments, revising the exposition, and restructuring the manuscript to
emphasize its main contributions. The author takes sole responsibility for
the mathematical statements, proofs, references, and conclusions.

%-----------------------------------------------------------------------
\clearpage
\bibliographystyle{amsplain}
\bibliography{suisse}

\end{document}